\documentclass[9pt]{amsart}
\usepackage{graphicx}
\usepackage{latexsym}
\usepackage{fancyhdr}
\usepackage{amsmath, amssymb}
 \usepackage[utf8]{inputenc}
\usepackage[all]{xy}
\usepackage{hyperref}

\usepackage{color}

\theoremstyle{plain}
\newtheorem{thm}{Theorem}

\theoremstyle{definition}

\theoremstyle{remark}

\numberwithin{equation}{section}

\newtheorem{ex}[thm]{Example}

\newcommand{\A}{\mathbb{A}}

\newcommand{\lgw}{\longrightarrow}

\newcommand{\ovl}{\overline}

\newcommand{\K}{\mathbb{K}}

\newcommand{\N}{\mathbb{N}}

\newcommand{\Q}{\mathbb{Q}}

\renewcommand{\a}{\alpha}

\renewcommand{\b}{\beta}

\let\mathscr\mathcal

\begin{document}
%
%%%%%
\title{Remark on a theorem of H. Hauser on textile maps}
\author{Guillaume Rond}
\email{guillaume.rond@univ-amu.fr}
\address{Aix Marseille Univ, CNRS, Centrale Marseille, I2M, Marseille, France}

\begin{abstract}
We give a counter example to the new theorem that appeared in the survey \cite{H} on Artin approximation. We then provide a correct statement and a proof of it.
\end{abstract}
\maketitle
\setlength{\parindent}{0cm}

The aforementioned theorem is the following one: \\
\\
\textbf{Strong Approximation Theorem for Textile Maps.} \cite{H} \emph{Assume that $\K$ is an algebraically closed field, and let $x$ denote again a vector of variables. Let $G : \K[[x]]^m \lgw \K[[x]]^q$ be a textile map. There exists an $l\in\N$ depending on $G$ such that, if $G = 0$ admits an approximate solution $y(x)\in \K[[x]]^m$ up to degree $l$,
$$G(y(x)) \equiv 0 \text{ mod }(x)^l,$$
 then there exists an exact solution $y(x) \in \K[[x]]^m$:
$G(y(x)) = 0.$}\\
 \\
Let us recall that the map $G$ is textile if $G(y(x))$ is a vector of power series whose coefficients are polynomials in the coefficients of $y(x)$.\\
This statement is incorrect as shown by the following example:\\
\begin{ex} We set $\K=\ovl\Q$.
Since $\ovl\Q$ is countable we may list its elements as $\a_0$, $\a_1$, \ldots, $\a_l$, \ldots.
Let $x$ be a single variable and let $G:\ovl\Q[[x]]\lgw\ovl\Q[[x]]$ be the textile map defined by
$$G\left(\sum_{k\geq 0}y_kx^k\right)=\sum_{l\geq 1}\left(\left(y_0-\a_{l-1}\right)y_l-1\right)x^l.$$
For every integer $N\geq 1$ let us set
$$y_N(x)=\a_N+\sum_{k=1}^{N}\frac{1}{\a_N-\a_{k-1}}x^k.$$
Then $G(y_N(x))\equiv 0$ modulo $(z)^{N+1}$. But there is no $y(x)=\sum_{k\geq 0}y_kx^k$ with $G(y(x))=0$. Indeed if such $y(x)$ would exist then we would have
\begin{equation}\label{eq}(y_0-\a_{l-1})y_l=1\ \ \ \forall l\geq 1.\end{equation}
But $y_0\in\ovl\Q$ so $y_0=\a_{l_0}$ for some $l_0\geq 0$. Thus \eqref{eq} for $l=l_0+1$ would give
$$0=(y_0-\a_{l_0+1-1})y_{l_0+1}=1$$
which is impossible.\end{ex}

In fact with  the additional assumption that $\K$ is uncountable the theorem is true:

\begin{thm}
The Strong Approximation Theorem for Textile Maps holds when $\K$ is a uncountable algebraically closed field.
\end{thm}

\begin{proof}
The proof given here is essentially Example 0.13 \cite{R}.\\
For a textile map $G:\K[[x]]^m\lgw \K[[x]]^q$ two situations may occur: either $G=0$ has an exact solution, either it does not. So the theorem  is equivalent to say if $G=0$ has no exact solution then there exists a $l_0$ such that $G$ has no approximate solution up to degree $l_0$. So the theorem is equivalent to say that if for every $l$ there exists $y(x)\in\K[[x]]^m$ such that $G(y(x))\in (x)^l$, then there exists an exact solution $y(x)\in\K[[x]]^m$: $G(y(x))=0$. This is what we are going to prove.\\
\\
Let us fix a textile map $G$ which has approximate solutions up to every degree and let $x=(x_1,\ldots,x_n)$. \\
The map $G$ being textile, for any $y(x)=(y_1(x),\ldots,y_m(x))$, with $y_i(x)=\sum_{\a\in\N^n}y_{i,\a}x^\a$, we have that
$$G(y(x))=\left(\sum_{\b\in\N^n} G_{1,\b}x^\b,\ldots, \sum_{\b\in\N^n} G_{q,\b}x^\b\right)$$
where the $G_{j,\b}$ are polynomials in the $y_{i,\a}$. For any $N\in\N$ let $D_N\geq N$ be an integer  such that the polynomials $G_{j,\b}$, for $|\b|< N$, depend only on the $y_{i,\a}$ for $|\a|< D_N$.\\
For every integer $k\geq l$ we consider the truncation maps:
$$\pi_k:\K[[x]]^m\lgw {\K[x]_{< k}}^m$$
$$\pi_{k,l}: {\K[x]_{< k}}^m\lgw {\K[x]_{< l}}^m$$
where $\K[x]_{< k}$ denote the set of polynomials in $x$ of degree $< k$. We identify ${\K[x]_{< k}}^m$ to the affine space $\A_\K^{m\binom{n+k-1}{k-1}}$ by identifying  a vector of polynomials with the vectors of the coefficients of these polynomials. With such an identification the maps $\pi_{k,l}$ correspond to projection maps.\\
For every positive integer $N$ the set
$$X_{N}:=\pi_{D_N}\left(\left\{y(x)\in\K[[x]]^m\mid G(y(x))\in (x)^N\right\}\right)$$
is exactly the affine variety 
$$\left\{y\in \mathbb{A}_\K^{m\binom{n+D_N-1}{D_N-1}}\mid G_{j,\b}(y)=0,\ \ j=1,\ldots,q,\  |\b|<N\right\}.$$
For every positive integers $N\geq k$ we define
$$C_N^k=\pi_{D_N,k}(X_N).$$
These sets are constructible subsets of $\A_\K^{m\binom{n+k-1}{k-1}}$ since $\K$ is algebraically closed (by Chevalley's Theorem). Let us recall that a constructible set is a finite union of sets of the form $W\backslash V$ where $W$ and $V$ are   Zariski closed  subsets of $\A_\K^{m\binom{n+k-1}{k-1}}$.\\
\\
Let us fix $k\geq 1$. Since $(x)^{N+1}\subset (x)^N$ for every positive integer $N$ and $\pi_{D_{N}}=\pi_{D_{N+1},D_N}\circ\pi_{D_{N+1}}$ we have that
$$C_{N+1}^k\subset C_N^k.$$
Thus the sequence $(C_N^k)_N$ is a decreasing sequence of constructible subsets of $\K^{m\binom{n+k-1}{k-1}}$. Let $F_N^k$ denote the Zariski closure of $C_N^k$. Then the sequence $(F_N^k)_N$ is a decreasing sequence of Zariski closed subsets of $\A_\K^{n\binom{n+k-1}{k-1}}$. By Noetherianity this sequence stabilizes, i.e. $F_{N}^k=F_{N_0}^k$ for every $N\geq N_0$ and some positive integer $N_0\geq k$. By assumption $C_{N_0}^k\neq \emptyset$  so $F_{N_0}^k\neq \emptyset$. Let $F$ be an irreducible component of $F_{N_0}^k$. \\
Since $C_N^k$ is constructible,  $C_N^k=\cup_iW^N_i\backslash V^N_i$ for a finite number of Zariski closed sets $W^N_i$ and $V^N_i$ with $W^N_i\backslash V^N_i\neq \emptyset$. So for $N\geq N_0$ we have that
$$F_{N_0}^k=F_N^k=\cup_i W^N_i.$$
 But $F$ being irreducible, for every $N\geq N_0$ one of the $W_i^N$ has to be equal to $F$. Thus for every $N\geq N_0$ there exists a closed proper subset $V_N\subset F$ such that
$$F\backslash V_N\subset C_N^k\ \ \forall N\geq N_0.$$
Since $\K$ is uncountable 
$$\bigcup_{N\geq N_0}V_N\subsetneq F.$$
This is a well known fact (see for instance Exercice 5.10, \cite{L} p. 76). This implies that 
$$A_k:=\bigcap_{N}C_N^k\neq \emptyset.$$
By definition $y^{(k)}\in A_k$ if and only if for every $l\geq k$ there exists $y^{(l)}\in {\K[x]_{<l}}^m$ such that $\pi_{l,k}(y^{(l)})=y^{(k)}$. In particular we have that
$$\pi_{l,k}(A_l)=A_k\ \ \ \forall l\geq k.$$
Thus for a given  $y^{(k)}\in A_k$, there exists an element $y^{(k+1)}\in A_{k+1}$ such that $\pi_{k+1,k}(y^{(k+1)})=y^{(k)}$. By induction there exists a sequence of $y^{(l)}\in A_l$, for every $l\geq k$, such that 
$$\pi_{l+1,l}(y^{(l+1)})=y^{(l)}.$$
At the limit we obtain $y\in\K[[x]]^m$ such that $G(y)\in (x)^l$ for every $l\geq k$, i.e. $G(y)=0$.
\end{proof}

\end{document}